\documentclass[12pt]{amsart}
\usepackage[utf8]{inputenc}
\usepackage[english]{babel}

\usepackage[left=25mm, right=25mm, top=25mm, bottom=25mm]{geometry}
\usepackage[shortlabels]{enumitem}
\usepackage{polynom, subcaption, hyperref, color, csquotes}
\hypersetup{
    colorlinks,
    citecolor=black,
    filecolor=black,
    linkcolor=black,
    urlcolor=black
}

\usepackage{graphicx, float, tikz-cd, tikz, caption}

\def\Ddots{\mathinner{\mkern1mu\raise\p@
\vbox{\kern7\p@\hbox{.}}\mkern2mu
\raise4\p@\hbox{.}\mkern2mu\raise7\p@\hbox{.}\mkern1mu}}
\makeatother

\usepackage{amsmath, amsfonts, amssymb, amsthm, mathtools, mathrsfs, commath}

\newtheorem {theorem}{Theorem}[section]
\newtheorem {lemma}[theorem]{Lemma}
\newtheorem {proposition}[theorem]{Proposition}

\newtheorem {corollary}[theorem]{Corollary}

\theoremstyle{definition}
\newtheorem {definition}[theorem]{Definition}
\newtheorem {remark}[theorem]{Remark}

\DeclareMathOperator\curl{curl}
\DeclareMathOperator\Div{div}
\renewcommand{\div}{\Div}

\begin{document}

\title{An equivariant Reeb-Beltrami correspondence and the Kepler-Euler flow}

\author{Josep Fontana-McNally}
\address{Josep Fontana-McNally,  
Laboratory of Geometry and Dynamical Systems, Department of Mathematics, Universitat Polit\`{e}cnica de Catalunya   }
\email{ josep.fontana@estudiantat.upc.edu}
\thanks{J. Fontana-McNally is supported by an UPC-INIREC grant through the grant ``Computational, dynamical and geometrical complexity in fluid dynamics”, Ayudas Fundación BBVA a Proyectos de Investigación Científica 2021. J. Fontana-McNally and E. Miranda are both partially supported by the Spanish State Research Agency grant PID2019-103849GB-I00 of AEI / 10.13039/501100011033 and by the AGAUR project 2021 SGR 00603 Geometry of Manifolds and Applications, GEOMVAP. All the authors are partially supported by the grant ``Computational, dynamical and geometrical complexity in fluid dynamics”, Ayudas Fundación BBVA a Proyectos de Investigación Científica 2021.}

\author{Eva Miranda}
\address{ Eva Miranda,
Laboratory of Geometry and Dynamical Systems and IMTech, Department of Mathematics, Universitat Polit\`{e}cnica de Catalunya  and Centre de Recerca Matemàtica, CRM }
\email{ eva.miranda@upc.edu}
\thanks{E. Miranda is supported by the Catalan Institution for Research and Advanced Studies via an ICREA Academia Prize 2021 and by the Alexander Von Humboldt foundation via a Friedrich Wilhelm Bessel Research Award. E. Miranda is also supported by the Spanish State
Research Agency, through the Severo Ochoa and Mar\'{\i}a de Maeztu Program for Centers and Units
of Excellence in R\&D (project CEX2020-001084-M). }

\author{Daniel Peralta-Salas} 
\address{Daniel Peralta-Salas, Instituto de Ciencias Matem\'aticas, Consejo Superior de Investigaciones Cient\'ificas,
28049 Madrid, Spain}
\email{dperalta@icmat.es}
\thanks{D. Peralta-Salas is supported by the grants CEX2019-000904-S, RED2022-134301-T and PID2019-106715GB GB-C21 funded by MCIN/AEI/10.13039/501100011033.}

\begin{abstract}
We prove that the correspondence between Reeb  and Beltrami vector fields presented in \cite{EG00} can be made equivariant whenever additional symmetries of the underlying geometric structures are considered.  As a corollary of this correspondence, we show that energy levels above the maximum of the potential energy of mechanical Hamiltonian systems can be viewed as stationary fluid flows, though the metric is not prescribed. In particular, we showcase the emblematic example of the $n$-body problem and focus on the Kepler problem. We explicitly construct a compatible Riemannian metric that makes the Kepler problem of celestial mechanics a stationary fluid flow (of Beltrami type) on a suitable manifold, the \textit{Kepler-Euler flow}. 
\end{abstract}

\maketitle



\section{Introduction}

Reeb and Beltrami vector fields are two important classes of vector fields that appear naturally in classical mechanics and hydrodynamics, respectively. In mechanical systems, Reeb vector fields are often obtained by restriction of Hamiltonian vector fields to level-sets of their Hamiltonians. Furthermore, the classical Weinstein conjecture asserts that Reeb vector fields exhibit at least one periodic orbit on compact manifolds\footnote{This has been proved in dimension 3 \cite{T07} and in multiple other scenarios}. On the other hand, Beltrami vector fields are in a sense (see Section \ref{sec:beltrami}), the most interesting stationary solutions of the incompressible Euler equations from a dynamical point of view. A more comprehensive introduction to Reeb and Beltrami vector fields is given in Subsections \ref{sec:contact} and \ref{sec:beltrami} below.  

 Sullivan envisaged\footnote{See \url{ https://www.math.stonybrook.edu/Videos/Einstein/425-19941109-Sullivan.html}} that these vector fields, as well as their associated geometric structures (contact forms for Reeb vector fields and Riemannian metrics for Beltrami vector fields), are closely interconnected. Etnyre and Ghrist formalized this idea in \cite{EG00}, showing that they are in fact related through a simple correspondence.

\begin{theorem}[\cite{EG00}]\label{thm:EG}
    Let $M$ be a $3$-dimensional manifold. For any contact form on $M$ and any nonzero rescaling of the associated Reeb vector field, there is a Riemannian metric for which the rescaled Reeb vector field is Beltrami. Conversely, for any Riemannian metric on $M$ and any nonvanishing Beltrami vector field, there is a contact form for which the Beltrami vector field is a rescaling of the Reeb vector field.
\end{theorem}

This correspondence has been used in various settings to translate results and techniques between contact geometry and hydrodynamics. On one hand, tools from contact geometry are adequate to construct Reeb vector fields with arbitrary topological complexity, so through the correspondence one obtains results about topological complexity of Beltrami fields for adapted Riemannian metrics. When the metric of the ambient space is fixed, analogous results are extremely hard to prove (see, e.g. \cite{EP18}). For instance, Etnyre and Ghrist used this technique to prove that there is a  non-vanishing $C^\omega$ Beltrami field on $\mathbb{S}^3$ for some adapted Riemannian metric which exhibits periodic flowlines of all possible knot and link types \cite{EG02}. Following a similar philosophy, the second and the third author of this article used the correspondence to construct a Beltrami field on $\mathbb{S}^3$ for some Riemannian metric, which is Turing complete (i.e., one can perform arbitrary computations by following streamlines of the flow) \cite{CMPP19}. On the other hand, analytical techniques on the hydrodynamics side were used to give lower bounds on the number of escape orbits of $b$-singular Reeb vector fields using a singular version of the correspondence \cite{MOP22, FMOP23}. 

However, in all these references the study of the symmetries of the systems is missing. In this article, we fill this gap and consider the effect of symmetries on both sides of the correspondence. In Section \ref{sec:equivariant} we extend Theorem \ref{thm:EG} by showing that we can also carry over any symmetries of the underlying geometric structures, that is, the correspondence is equivariant.

\begin{theorem}\label{thm:equivariantmirror}
Let $M$ be a $(2n+1)$-dimensional manifold and $\rho: G\times M\rightarrow M$ a \emph{compact} Lie group action on $M$. For any non-vanishing $\rho$-invariant Beltrami vector field $X$ with $\rho$-invariant Riemannian metric $g$, there is a $\rho$-invariant contact form for which $X$ is a rescaling of the corresponding Reeb vector field. Conversely, for each $\rho$-invariant rescaling of a Reeb vector field $X$ with contact form $\alpha$, there is a $\rho$-invariant Riemannian metric for which $X$ is Beltrami.
\end{theorem}

This result is particularly relevant when studying physical systems, whose symmetries are key in their treatment. In Section \ref{sec:kepler} we give a nice example of this equivariant correspondence, showing how it can be used to view the classical Kepler problem of celestial mechanics -- the motion of a satellite orbiting a planet -- along with its symmetries, as a stationary Beltrami solution to the Euler equations on an adapted Riemannian manifold. We call this fluid flow the Kepler-Euler flow. The example is based on the well-known Moser-Osipov-Belbruno regularization of the Kepler problem (see Theorem \ref{thm:Moser}), which provides a way to rescale time to regularize collisions (i.e., when the satellite collides with the planet at the origin) without otherwise affecting the dynamics of the system. After mapping the flow on energy levels of the Kepler problem to adequate manifolds (shown in Figure \ref{fig:stereographicprojections1} below), we can state the result about the Kepler-Euler flow as follows. 

\begin{theorem}[The Kepler-Euler flow]\label{thm:KeplerEuler}
The regularized Kepler flow on the $c$-energy level is a stationary Beltrami solution to the Euler equations on
\begin{itemize}
\item $S^*\mathbb{S}^2_c$ if $c < 0$,
\item $S^*\mathbb{R}^2$ if $c = 0$ and
\item $S^*\mathbb{H}^2_c$ if $c > 0$,
\end{itemize}
where $S^*$ denotes the spherical cotangent bundle. The Riemannian metrics are the lifts to the spherical cotangent bundles of the natural constant $(-2c)$-curvature metrics, and the symmetries of the Kepler problem correspond to isometries of the Riemannian metrics.
\end{theorem}

\begin{remark}
In the theorem above, the flow lines are lifted geodesics. The Kepler flow on the plane is recovered from the natural stereographic projections of the respective surfaces, or from the involution $x\mapsto\frac{2x}{|x|^2}$ when $c=0$, as shown in Figure \ref{fig:stereographicprojections1}. We note that the fact that symmetries of the Kepler problem correspond to symmetries of the respective metrics was already well-known since Moser's regularization.
\end{remark}
\begin{figure}[H]
    \centering
    \input{StereographicProjections.tex}
    \caption{The maps which induce the Riemannian metrics whose geodesic flow is the regularized Kepler flow.}
    \label{fig:stereographicprojections1}
\end{figure}

This new perspective on the Kepler problem evokes the metaphoric image of a fluid moving planets and stars through space, but it also raises the more \textit{down to Earth} question, \textit{what systems of classical mechanics are solutions to fluid equations?} Indeed, at the end of Section \ref{sec:equivariant} we use Theorem \ref{thm:equivariantmirror} to obtain the following corollary.

\begin{corollary}
    The flow of any mechanical Hamiltonian system (i.e., a system where the Hamiltonian is the sum of kinetic and potential energies) on an energy level above the maximum value of the potential is a stationary Beltrami solution to the Euler equations on the energy level with some adapted Riemannian metric. Symmetries of the system translate to isometries of the metric.
\end{corollary}
In particular, we can reinterpret the $n$-body problem on positive energies as a stationary fluid flow.
\begin{corollary}
    The flow of the $n$-body problem on a positive energy level is a stationary Beltrami solution to the Euler equations for some adapted Riemannian metric.
\end{corollary}
We remark that the metrics that make Beltrami the aforementioned Hamiltonian systems on some energy levels are not prescribed a priori, and it is generally very hard to understand their curvature.

In the next section we briefly introduce Reeb and Beltrami vector fields and proceed with the proofs of Theorems \ref{thm:equivariantmirror} and \ref{thm:KeplerEuler} in Sections \ref{sec:equivariant} and \ref{sec:kepler} below.
\section{Basic results}
 \subsection{Contact geometry and Reeb vector fields}\label{sec:contact}
 
 Recall that a symplectic form on an even-dimensional manifold is a closed, nondegenerate differential 2-form $\omega$, and that a Hamiltonian system on a symplectic manifold $(M,\omega)$ with Hamiltonian function $H\in C^\infty(M,\mathbb{R})$ is given by the vector field $X_H$ such that
\begin{equation*}
	\iota_{X_H} \omega = \omega(X_H, \cdot) = -dH.  
\end{equation*}
For example, the Kepler problem is a Hamiltonian system on $(T^*(\mathbb{R}^2\setminus \{0\}),\omega)$, where $\omega = dp_1\wedge dq_1 + dp_2\wedge dq_2$ is the standard symplectic form, and the Hamiltonian function is given by
\begin{equation}
    H(q,p) = \frac{|p|^2}{2} - \frac{1}{|q|}.
\end{equation}

One of the fundamental problems when studying a Hamiltonian system is to describe its periodic orbits. For example, for the $3$-body problem of celestial mechanics (the motion of three massive bodies interacting through gravitational force), only in very particular cases are periodic orbits known to exist (for a thorough exposition of work on the three-body problem since Poincaré we recommend \cite{C15}). Since energy is conserved along the flow of a Hamiltonian system, we can immediately restrict our search for periodic orbits to energy levels of the Hamiltonian to reduce the system by one dimension. 

Reeb vector fields come into play when energy levels carry a contact form which is a primitive of the symplectic form.

\begin{definition}[Contact form]\label{def:contact}
A \emph{contact form} on a $(2n+1)$-dimensional manifold $\Sigma$ is a differential 1-form $\alpha$ such that $\alpha\wedge (d\alpha)^{n} \neq 0$, or equivalently,
\begin{equation*}
T \Sigma = \ker\alpha \oplus \ker d\alpha.
\end{equation*}
The \emph{Reeb vector field} associated to $\alpha$ is the unique vector field $R$ on $\Sigma$ such that
\begin{itemize}
\item $\langle R\rangle = \ker d\alpha$ and
\item $\alpha(R) = 1$.
\end{itemize}
We say that a vector field $X$ is \emph{Reeb-like}  if $X = fR$ for some smooth factor $f>0$.
\end{definition}

If a contact form on a regular energy level $\Sigma = H^{-1}(c)$ is a primitive of the symplectic form, that is, $d\alpha = \omega_{|\Sigma}$, then the Reeb vector field is parallel to the Hamiltonian vector field $X_H$ on the energy level,
\begin{equation*}
  \langle R\rangle = \ker d\alpha = \ker \omega_{|\Sigma} = \langle X_H\rangle.
\end{equation*}
Therefore, the Hamiltonian vector field has periodic orbits if and only if the Reeb vector field does. This is relevant because Weinstein's conjecture states that every Reeb vector field on a closed manifold has a closed orbit, so contact geometry and Reeb vector fields can be useful to prove existence of periodic orbits of Hamiltonian systems. This is precisely what is done in \cite{AFKP12} to prove that the regularized circular, planar, restricted three-body problem has periodic orbits on compactified energy levels below the first Lagrange point $L_1$.

When energy levels are not compact, saying something interesting about the dynamics becomes more complicated. One can either resort to results such as Berg-Pasquotto-Vandervorst's  \cite{BPV09}, or compactify the energy level, somehow extending the dynamics in a physically meaningful way. This last option, especially when compactifying unbounded energy levels with contact forms, often gives rise to singularities in the extended contact form along the added boundary. This is the case, for example, when positive energy levels of the restricted 3-body problem are compactified using a McGehee change of coordinates \cite{MOP22}. Studying singular contact forms and finding an equivalent to Weinstein's conjecture in this setting thus becomes an important question. In this line of research, we have shown that for a generic set within a class of so-called $b$-contact forms (contact forms having a logarithmic singularity along a hypersurface), one can give lower bounds on the number of escape orbits of the system \cite{MOP22, FMOP23}.

\subsection{Hydrodynamics and Beltrami vector fields}\label{sec:beltrami} 

We now move on to a subject at first glance completely unrelated to celestial mechanics, hydrodynamics. The mathematical model of a fluid is that it is composed of \emph{fluid particles}, each with a certain velocity which can be jointly written as a vector field $X$ on a smooth manifold $M$ in which the fluid is contained. We call this vector field $X$ the \emph{velocity field}, and it can depend on time and position on $M$, so $X=X(t,x)$. Fluid equations are equations that such a velocity field must satisfy, and the simplest such equations are the \emph{Euler equations}. To formulate them, a Riemannian metric $g$ and a distinguished volume form $\mu$ on $M$, which for most physical applications is the Riemannian volume form, are required. We assume that the density of the fluid is constant and equal to one. Applying Newton's second law to the trajectory of a fluid particle $x(t)$ we get
\begin{equation*}
\frac{d^2}{dt^2}x(t) = F \Rightarrow \frac{d}{dt}X(t,x(t)) = \frac{\partial X}{\partial t} + \nabla_X X = F,
\end{equation*}
where  $\nabla_X$ is the Levi-Civita connection given by $g$ and $F$ is the force acting on the particle. The incompressible Euler equations are obtained assuming that the only force acting on the fluid particles comes from the internal pressures of the fluid, so that if $P$ is the pressure at each point in the fluid, $F=-\nabla P$. Furthermore, it is assumed that the fluid is incompressible, which translates to $\div (X)\mu = \mathcal{L}_{X}\mu = 0$. In this article we are concerned with the special case in which the velocity field $X$ is stationary, so $\frac{\partial X}{\partial t}=0$. The stationary incompressible Euler equations are 
\begin{equation*}
\begin{cases}
\nabla_X X = -\nabla P \\
\mathcal{L}_{X} \mu = 0 .
\end{cases}
\end{equation*}

It is convenient to reformulate the stationary Euler equations in the language of differential forms, see e.g. \cite{AK98}. To do so, we dualize the equations by the Riemannian metric, thus obtaining
\begin{equation}\label{eq:eulereqs}
\begin{cases}
\iota_X d\iota_X g = -dB\\
\mathcal{L}_{X} \mu = 0
\end{cases},
\end{equation}
where $B = P + \frac{1}{2}||X||^2$ is known as the \emph{Bernoulli function}. The last remaining object to be introduced in this section is the \emph{curl} of a vector field, which in Euclidean $\mathbb{R}^3$ is the usual curl operator,  $\curl X = \nabla\times X$. On general Riemannian manifolds and in the language of differential forms, the curl operator is generalized to 

\begin{definition}[Curl operator]
Given a $(2n+1)$-dimensional Riemannian manifold with distinguished volume form $(M,g,\mu)$, the \emph{curl} of a vector field $X$ with respect to $g$ and $\mu$ is the unique vector field $\curl X$ that satisfies
\begin{equation*}
\iota_{\curl  X}\mu = (d\iota_X g)^{n}.
\end{equation*}
\end{definition}

When $X$ is a velocity field, $\curl X$ is commonly known as the \textit{vorticity field} of the flow. A particularly relevant class of stationary solutions to the Euler equations are velocity fields that (in addition to being divergence-free) are parallel to their vorticity field: $\curl X = fX$. In this case, the Bernoulli function is constant $B=c$.

\begin{definition}[Beltrami vector field]
A divergence-free (with respect to the volume form~$\mu$) vector field $X$ on $(M,g,\mu)$ is a \emph{Beltrami vector field} if $\curl  X = fX$. We call a Beltrami field \emph{nonsingular} if neither it nor $f$ vanish at any point. 
\end{definition}

The following proposition is standard and shows that there is a natural contact form associated to a Beltrami field. We expand on this observation in the next section. 

\begin{proposition}\label{prop:beltramicontact}
    A divergence-free nonsingular Beltrami vector field $X$ on $(M,g,\mu)$ is a solution to the stationary Euler equations. Furthermore, $\iota_X g$ is a contact form on $M$.
\end{proposition}
\begin{proof}
  Since $X$ is Beltrami, noticing that $\mu = h\mu_g$ for some factor $h>0$ ($\mu_g$ is the Riemannian volume), we have
    \begin{equation*}
     \iota_X g \wedge (d\iota_X g)^n = f\iota_X g\wedge\iota_X \mu = fh\iota_X g\wedge\iota_X \mu_g = fh||X||^2\mu_g\neq 0.
    \end{equation*}
   This means that $\iota_X g$ is a contact form. Moreover
        \begin{equation}\label{eq:xcontact}
    \iota_X (d\iota_X g)^n = \iota_X\iota_{\curl  X}\mu =           \iota_X\iota_{fX}\mu = 0,
    \end{equation}
and since $\iota_Xg$ is a contact form, it is obvious that $X$ is in the kernel of $d\iota_Xg$, i.e.,  $\iota_X d\iota_X g = 0$. This shows that $X$ is a solution to the stationary Euler equations (\ref{eq:eulereqs}) on $(M,g,\mu)$ with constant Bernoulli function.
\end{proof}

\begin{remark}
Assume that $\curl X = fX$ for some nonvanishing factor $f$ and volume form $\mu$, and we do not assume that $X$ preserves the volume form $\mu$. Then 
    \begin{equation}\label{eq:xdivergence}
        \mathcal{L}_X \mu = d\iota_X \mu = d(\frac{1}{f}\iota_{\curl X}\mu) = d(\frac{1}{f}(d\iota_Xg)^n) = d\frac{1}{f}\wedge (d\iota_X g)^n,
    \end{equation}
which is not necessarily zero. However, since the proportionality function $f$ is nonvanishing, it is obvious that $X$ preserves the volume form $\tilde{\mu} = f\mu$. Additionally, if we denote by $\widetilde{\curl}$ the curl operator computed with the volume form $\tilde\mu$, it is elementary to check that
    \begin{equation*}
         \widetilde{\curl }X = X.
    \end{equation*}
    Accordingly, $X$ is a Beltrami field with constant proportionality factor for the volume form~$\tilde\mu$.
    
\end{remark}

For a much more comprehensive introduction to stationary Euler flows, we recommend \cite{A66, AK98, EP18}.


\section{The Equivariant Correspondence}\label{sec:equivariant}
In this section we show that Etnyre and Ghrist's correspondence between Reeb and Beltrami vector fields of Theorem \ref{thm:EG} can be made to preserve symmetries of the metrics or contact forms.

\begin{theorem}[The equivariant correspondence]\label{thm:equivariantcorrespondence}
Let $M$ be a $(2n+1)$-dimensional smooth manifold and $\rho : G \times M \rightarrow M$ a compact Lie group action on $M$. For each $\rho$-invariant nonsingular Beltrami field $(X,g)$ there is a $\rho$-invariant contact form for which $X$ is Reeb-like. Conversely, for each $\rho$-invariant Reeb-like field $(X,\alpha)$ there is a $\rho$-invariant Riemannian metric for which $X$ is nonsingular Beltrami. 
\end{theorem}

When we refer to Reeb and Beltrami vector fields, we refer to the vector field \emph{and the subjacent geometric structure}, so that when we say that $(X,g)$ is $\rho$-invariant, for example, we mean that $\rho_{\sigma*}X = X$ and $\rho_\sigma^*g = g$ for every $\sigma\in G$. 

Before continuing with the proof, we recall the definitions of a couple of objects that will be used therein. The \textit{Haar measure} on a compact Lie group $G$ is the unique measure $\eta$ on $G$ such that $\eta$ is left invariant and $\eta(G)=1$.

\begin{remark}
Given a compact Lie group action $\rho : G\times M \rightarrow M$, it is standard that the Haar measure provides a way to \emph{average} tensor fields on $M$ to make them invariant by the action. Indeed, if $T$ is a tensor field on $M$, the tensor field defined pointwise as
\begin{equation*}
\tilde{T}_p = \int_{\sigma\in G} (\rho_\sigma^*T)_p d\eta
\end{equation*}
is a  $\rho$-invariant tensor field. The integral is taken over all $\sigma\in G$ and with respect to the Haar measure.
\end{remark}

In particular, we will use this averaging procedure for Riemannian metrics. Since the space of positive definite symmetric bilinear forms is convex, it is well known that averaging a Riemannian metric by the procedure above yields a Riemannian metric that is invariant by the group action.

Recall, also, that an almost complex structure on a vector bundle $E\rightarrow M$ over a smooth manifold $M$ is a section $J:M\rightarrow End(E)$ such that $J^2 = -Id$. Almost complex structures establish a deep connection between complex and symplectic geometry. However, the only property we will use is that they provide a way to construct Riemannian metrics from nondegenerate 2-forms: For every nondegenerate differential 2-form $\omega$ on a vector bundle $E\rightarrow M$ there is an almost complex structure $J$ such that $g(\cdot,\cdot) = \omega(\cdot, J\cdot)$ is a positive definite, symmetric bilinear form on $E$ (see \cite{MS17}).

\begin{proof}[Proof of Theorem \ref{thm:equivariantcorrespondence}]
Let $(X,g)$ be a $\rho$-invariant nonsingular Beltrami field. The construction of a $\rho$-invariant contact form is the one in the proof of Proposition \ref{prop:beltramicontact}. Invariance of $X$ and $\alpha$ are automatic because $X$ was already taken to be invariant and $\alpha$ is a contraction of the invariant metric $g$ with an invariant vector field. Let us now prove the equivariance of the converse direction.

Let $(X,\alpha)$ be a $\rho$-invariant Reeb-like field, with $X=hR$, where $R$ is the Reeb field with respect to $\alpha$ and $h>0$. Since $X$ is $\rho$-invariant, $h$ must also be so, as $R$ is invariant by the invariance of $\alpha$. Using the splitting of $TM$ into $\ker \alpha\oplus \ker d\alpha$, we can define a Riemannian metric on each component and impose mutual orthogonality to naturally construct a Riemannian metric following \cite{G07}. Consider 
\begin{equation*}
\tilde{g} = \frac{1}{h}\alpha\otimes\alpha + d\alpha(\cdot,J\cdot),
\end{equation*}
where $J$ is an almost complex structure on the vector bundle $\ker\alpha \rightarrow M$ (the complex structure) making $d\alpha(\cdot,J\cdot)$ a Riemannian metric on this bundle. Note that in general $\tilde{g}$ is not $\rho$-invariant because $J$ is not. Nevertheless, we can average this metric using the Haar measure on $G$ to obtain a $\rho$-invariant metric
\begin{equation*}
g = \int_G \rho^*(\tilde{g})d\eta = \int_G \rho^*(\frac{1}{h}\alpha\otimes\alpha)d\eta + \int_G \rho^*(d\alpha(\cdot,J\cdot)) d\eta = \frac{1}{h}\alpha\otimes\alpha + \int_G \rho^*(d\alpha(\cdot,J\cdot)) d\eta,
\end{equation*}
where we have used that the first summand is $\rho$-invariant. 

We claim that $X$ is Beltrami with respect to $g$. Indeed, the contraction of $X$ with the second summand of the metric vanishes, since $X\in \ker d\alpha$, and therefore
\begin{equation*}
\iota_X g = \frac{1}{h}\alpha(hR)\alpha = \frac{h}{h}\alpha(R)\alpha = \alpha.
\end{equation*}
Taking the volume form $\mu = \frac{1}{h}\alpha\wedge (d\alpha)^{n}\neq 0$ we finally arrive at
\begin{equation*}
(d\iota_X g)^{n} = (d\alpha)^{n} = \iota_{X} \mu,
\end{equation*}
and therefore $\curl X = X$. Also notice that for the chosen $\mu$,
\begin{equation*}
    \mathcal{L}_X\mu = d\iota_X\mu = d(d\alpha)^{n} = 0,
\end{equation*}
so $X$ is also divergence-free and therefore $X$ is Beltrami with respect to $(g,\mu)$. This completes the proof of the theorem.
\end{proof}

An interesting consequence of the Reeb-Beltrami correspondence is that it allows us to view the $n$-body problem of celestial mechanics on positive energy levels as a Beltrami vector field on some Riemannian manifold. 

\subsection{Mechanical Hamiltonian systems as fluid flows and the $n$-body problem}\label{sec:nbodyfluid}

The $n$-body problem of celestial mechanics is the problem of determining the dynamics of a system of $n$ bodies in $\mathbb{R}^d$ moving according to classical laws of gravitation. This system can be expressed in the Hamiltonian formalism as follows. Take canonical coordinates $(q_1,\dots,q_n,p_1,\dots,p_n)$ on $T^*\mathbb{R}^{dn}$ with the canonical symplectic form $\omega = dp\wedge dq$, where $q_i$ represents the position of the $i$th body and $p_i$ its momentum. Denoting by $m_i$ the mass of the $i$th body, the Hamiltonian describing the system is
\begin{equation}
    H(q,p) = K(p) + U(q),
\end{equation}
with $K(p) = \sum_{1\leq i\leq n} \frac{|p_i|^2}{2m_i}$ the kinetic energy and $U(q) = -\sum_{1\leq i < j \leq n}\frac{m_im_j}{|q_i - q_j|}$ the potential energy. The Hamiltonian is a \textit{mechanical Hamiltonian} (i.e. consisting of a sum of kinetic and potential energies depending on momenta and positions respectively). 

\begin{lemma}[Lemma 2.6.3 and Remark 2.6.5 of \cite{FK18}]\label{lemma:mechanicalcontact}
    Let $M$ be a smooth manifold and $H = K(p) + U(q)$ a mechanical Hamiltonian on $T^*M$, where $K(p) = \frac{|p|^2_g}{2}$ for some Riemannian metric $g$ on $M$. Then for $c > \max U$, the Hamiltonian vector field $X_H$ restricted to the energy level $\Sigma = H^{-1}(c)$ is Reeb-like with respect to the canonical Liouville form $\alpha = pdq$ restricted to $\Sigma$.
\end{lemma}
\begin{proof}
    As in the discussion of Subsection \ref{sec:contact}, on one hand we have
    \begin{equation*}
        \omega|_\Sigma = d\alpha|_\Sigma \Rightarrow \langle X_H\rangle = \ker d\alpha|_\Sigma.
    \end{equation*}
    Furthermore, $\alpha$ is contact because $\alpha(X_H) = K(p) > 0$ on $\Sigma$, so 
    \begin{equation*}
        \ker d\alpha|_\Sigma \cap \ker \alpha|_\Sigma = 0,
    \end{equation*}
    which yields the contact condition.

\end{proof}

An immediate consequence of the above lemma and Theorem \ref{thm:equivariantcorrespondence} is the following reinterpretation of the $n$-body problem as a stationary incompressible fluid flow on some Riemannian manifold.

\begin{corollary}\label{cor:nbodyasfluid}
    The flow of the $n$-body problem on a positive energy level is a stationary Beltrami solution to the Euler equations on a hypersurface of $T^*\mathbb{R}^{dn}$ with some Riemannian metric. Any symmetries of the Hamiltonian field correspond to isometries of the corresponding metric.
\end{corollary}

A natural question in view of this result is if the metric on an energy level which makes the Hamiltonian flow Beltrami can be characterized in any way. A natural guess for such a metric is the canonical lift of the  base metric to the cotangent bundle, restricted to the energy level. In the next section we show that this is the case when the Hamiltonian is  purely kinetic energy, but it is not true in general. In the next section we also present an explicit example of Corollary \ref{cor:nbodyasfluid}, namely we obtain an explicit formulation of the Kepler problem as an Euler flow.


\section{The Kepler-Euler Flow}\label{sec:kepler}
The $2$-body problem can be reduced to the Kepler problem, which is that of describing the dynamics of a system comprised of a large body, which we  call the \emph{star} and assume centered at the origin of the plane $\mathbb{R}^2$, and a body of comparatively negligible mass orbiting around the first, which we call the \emph{planet}. In particular, since the planet exerts negligible force on the star, the star remains stationary and therefore we need only to describe the motion of the planet. Taking canonical coordinates $(q,p) = (q_1,q_2,p_1,p_2)$ on $(T^*\mathbb{R}^2, \omega = p_1\wedge q_1 + p_2\wedge q_2)$, the Hamiltonian describing the system (ignoring masses and other constants) is
\begin{equation*}
H:T^*(\mathbb{R}^2\setminus\{0\}) \rightarrow \mathbb{R}, \quad (q,p) \mapsto \frac{|p|^2}{2} - \frac{1}{|q|}.
\end{equation*}
By Corollary \ref{cor:nbodyasfluid}, the Kepler problem on a positive energy level is a solution to the Euler equations for some Riemannian metric. Our goal is to find such a metric. We recall that a metric for which a Reeb-like vector field is Beltrami is called an \emph{adapted metric} to the Reeb-like field. Adapted metrics were introduced by Chern and Hamilton in \cite{CH85}, though, as mentioned in the introduction, it was Sullivan who first understood the connection with Beltrami fields in hydrodynamics.

\subsection*{Regularizing the problem}
When the planet is moving directly towards the star, a \textit{collision} occurs and the planet's velocity blows up in finite time. The flow is therefore not complete and it is advantageous to reparametrize time to fix this issue. Since we are mainly interested in the qualitative behaviour of trajectories, reparametrizing time is not a problem and we will consider two orbits to be identical if they are so up to reparametrization. This process is known as \emph{regularization} of the system and it is standard in dynamical systems. To regularize the Kepler problem we follow Arnold in \cite{AKN06}. We include a proof for the sake of completeness.

\begin{lemma}
Let $\gamma(t)$ be an integral curve of energy $c$ of a Hamiltonian $H$. If we reparametrize time by $\tau\mapsto t(\tau)$ where $\frac{dt}{d\tau} = G(x)$, then $\gamma(\tau) = \gamma(t(\tau))$ is an integral curve of energy 0 for the Hamiltonian $\bar{H} = G(H-c)$. If we take $G = (H + k)$ for a constant $k$, we can take $\bar{H} = \frac{1}{2}(H + k)^2$ and $\gamma(\tau)$ will have energy $\frac{(c+k)^2}{2}$ instead.
\end{lemma}
\begin{proof}
We must check that $\frac{d\gamma}{d\tau}(\tau) = X_{\bar{H}}(\gamma(\tau))$. Applying the chain rule we have
\begin{equation*}
\frac{d\gamma}{d\tau}(\tau) = \frac{d\gamma}{dt}\frac{dt}{d\tau} = X_HG.
\end{equation*}
Now,
\begin{equation*}
\omega(\frac{d\gamma}{d\tau}(\tau),\cdot) = G\omega(X_H,\cdot) = -GdH = -(GdH + (H-c)dG) = -d(G(H-c)) = -d\bar{H},
\end{equation*}
where we have used that $(H-c) = 0$ along $\gamma$. From this we see that $\frac{d\gamma}{d\tau}(\tau) = X_{\bar{H}}(\gamma(\tau))$. If $G = (H + k)$, we have
\begin{equation*}
-GdH = -(H+k)dH = -d\left(\frac{1}{2}(H + k)^2 \right)= -d\bar{H},
\end{equation*}
as desired.
\end{proof}

With the help of this lemma we reparametrize time, slowing it down as the planet approaches the star taking $G = |q|$. On an energy level $c$, the regularized Kepler Hamiltonian becomes
\begin{equation*}
\bar{H}(q,p) = G(H-c) = |q|(\frac{|p|^2}{2} - c) - 1.
\end{equation*}

For reasons that will shortly become clear it is convenient to apply the lemma again with $G = (\bar{H} + 1)$ to obtain another equivalent Hamiltonian
\begin{equation*}
K_c(q,p) = \frac{|q|^2}{2}\Big(\frac{|p|^2 - 2c}{2}\Big)^2,
\end{equation*}
with integral curves on the energy level $K_c=\frac{1}{2}$ identical to integral curves of the original Hamiltonian on the $c$-energy level up to reparametrization. We henceforth refer to the flow of $K$ as the \emph{regularized Kepler flow} on the $c$-energy level, omitting $c$ when it does not lead to confusion. This Hamiltonian is particularly interesting because after a symplectic coordinate change called \emph{symplectic switch} given by $(x,y) \mapsto (q,p)=(y, -x)$, the Hamiltonian becomes
\begin{equation}\label{eq:regularizedKepler}
K(x,y) = \frac{|y|^2}{2}\Big(\frac{|x|^2 - 2c}{2}\Big)^2 = \frac{|y|^2_{g^*}}{2},
\end{equation}
where  $g^*$ is the inverse of the conformally flat metric $g = (\frac{2}{|x|^2 - 2c})^2\langle\cdot,\cdot\rangle_{euc}$. The physical interpretation is that $K$ consists only of kinetic energy. Such Hamiltonians are called \emph{kinetic} and they are of particular interest because their trajectories are lifted geodesics of the metric \cite{FK18}. Physically, since there is no potential, there is no external force acting on the system. It easy to check that $g$ is a constant $-2c$ curvature metric, and that the \emph{stereographic projections} shown in the following figure are the maps that induce these metrics on $\mathbb{R}^2$. When $c<0$ the variable $x$ takes values on the whole $\mathbb R^2$, when $c=0$ $x\in\mathbb R^2\backslash\{0\}$ and for $c>0$ the configuration space is $\{x\in\mathbb R^2: |x|^2>2c\}$.

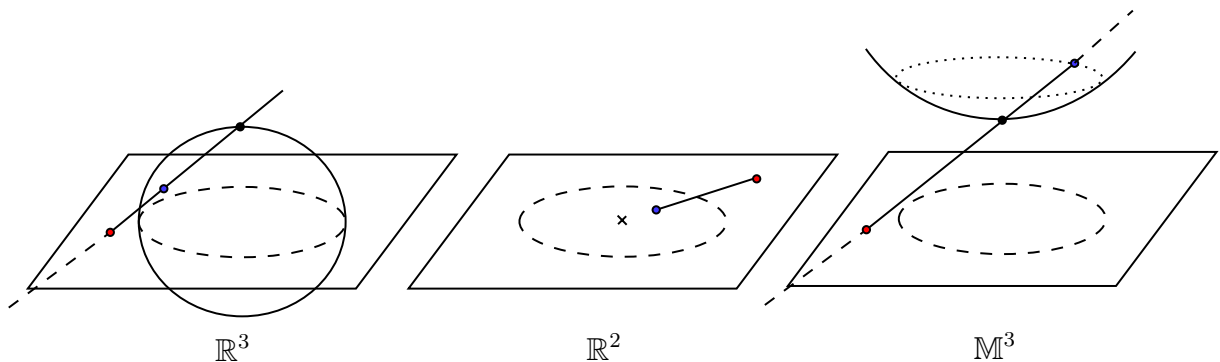
\begin{center}

\tikzset{every picture/.style={line width=0.75pt}} 

\begin{tikzpicture}[x=0.75pt,y=0.75pt,yscale=-1,xscale=1]

\draw   (145.75,110.33) -- (311.25,110.33) -- (260.5,178) -- (95,178) -- cycle ;
\draw   (337.75,110.33) -- (503.25,110.33) -- (452.5,178) -- (287,178) -- cycle ;
\draw   (529.08,109) -- (694.58,109) -- (643.83,176.67) -- (478.33,176.67) -- cycle ;
\draw   (150.96,144.17) .. controls (150.96,117.75) and (174.31,96.33) .. (203.13,96.33) .. controls (231.94,96.33) and (255.29,117.75) .. (255.29,144.17) .. controls (255.29,170.58) and (231.94,192) .. (203.13,192) .. controls (174.31,192) and (150.96,170.58) .. (150.96,144.17) -- cycle ;
\draw  [dash pattern={on 4.5pt off 4.5pt}] (150.96,144.42) .. controls (150.96,134.61) and (174.31,126.67) .. (203.13,126.67) .. controls (231.94,126.67) and (255.29,134.61) .. (255.29,144.42) .. controls (255.29,154.22) and (231.94,162.17) .. (203.13,162.17) .. controls (174.31,162.17) and (150.96,154.22) .. (150.96,144.42) -- cycle ;
\draw    (136.6,149.6) -- (223.67,78) ;
\draw  [dash pattern={on 4.5pt off 4.5pt}]  (85.4,187.6) -- (136.6,149.6) ;
\draw  [fill={rgb, 255:red, 255; green, 0; blue, 0 }  ,fill opacity=1 ] (134.77,149.6) .. controls (134.77,148.59) and (135.59,147.77) .. (136.6,147.77) .. controls (137.61,147.77) and (138.43,148.59) .. (138.43,149.6) .. controls (138.43,150.61) and (137.61,151.43) .. (136.6,151.43) .. controls (135.59,151.43) and (134.77,150.61) .. (134.77,149.6) -- cycle ;
\draw  [fill={rgb, 255:red, 0; green, 0; blue, 0 }  ,fill opacity=1 ] (200.29,96.33) .. controls (200.29,95.32) and (201.11,94.5) .. (202.13,94.5) .. controls (203.14,94.5) and (203.96,95.32) .. (203.96,96.33) .. controls (203.96,97.35) and (203.14,98.17) .. (202.13,98.17) .. controls (201.11,98.17) and (200.29,97.35) .. (200.29,96.33) -- cycle ;
\draw  [fill={rgb, 255:red, 64; green, 54; blue, 255 }  ,fill opacity=1 ] (161.77,127.6) .. controls (161.77,126.59) and (162.59,125.77) .. (163.6,125.77) .. controls (164.61,125.77) and (165.43,126.59) .. (165.43,127.6) .. controls (165.43,128.61) and (164.61,129.43) .. (163.6,129.43) .. controls (162.59,129.43) and (161.77,128.61) .. (161.77,127.6) -- cycle ;
\draw   (392.47,141.14) -- (397.03,145.77)(396.8,141.44) -- (392.7,145.48) ;
\draw    (517.67,57) .. controls (549,101) and (613,107) .. (653.67,58.33) ;
\draw  [dash pattern={on 4.5pt off 4.5pt}] (534.29,142.83) .. controls (534.29,133.03) and (557.65,125.08) .. (586.46,125.08) .. controls (615.27,125.08) and (638.63,133.03) .. (638.63,142.83) .. controls (638.63,152.64) and (615.27,160.58) .. (586.46,160.58) .. controls (557.65,160.58) and (534.29,152.64) .. (534.29,142.83) -- cycle ;
\draw  [draw opacity=0][dash pattern={on 0.84pt off 2.51pt}] (532.95,71.34) .. controls (532.95,71.34) and (532.95,71.34) .. (532.95,71.34) .. controls (532.95,71.34) and (532.95,71.34) .. (532.95,71.34) .. controls (532.98,65.63) and (556.29,61.1) .. (585.03,61.23) .. controls (613.76,61.35) and (637.03,66.08) .. (637.01,71.8) .. controls (636.98,77.51) and (613.67,82.04) .. (584.94,81.91) .. controls (557.09,81.79) and (534.37,77.34) .. (533.02,71.87) -- (584.98,71.57) -- cycle ; \draw  [dash pattern={on 0.84pt off 2.51pt}] (532.95,71.34) .. controls (532.95,71.34) and (532.95,71.34) .. (532.95,71.34) .. controls (532.95,71.34) and (532.95,71.34) .. (532.95,71.34) .. controls (532.98,65.63) and (556.29,61.1) .. (585.03,61.23) .. controls (613.76,61.35) and (637.03,66.08) .. (637.01,71.8) .. controls (636.98,77.51) and (613.67,82.04) .. (584.94,81.91) .. controls (557.09,81.79) and (534.37,77.34) .. (533.02,71.87) ;  
\draw    (517.93,148.27) -- (622.93,64.27) ;
\draw  [dash pattern={on 4.5pt off 4.5pt}]  (466.73,186.27) -- (517.93,148.27) ;
\draw  [fill={rgb, 255:red, 255; green, 0; blue, 0 }  ,fill opacity=1 ] (516.1,148.27) .. controls (516.1,147.25) and (516.92,146.43) .. (517.93,146.43) .. controls (518.95,146.43) and (519.77,147.25) .. (519.77,148.27) .. controls (519.77,149.28) and (518.95,150.1) .. (517.93,150.1) .. controls (516.92,150.1) and (516.1,149.28) .. (516.1,148.27) -- cycle ;
\draw  [fill={rgb, 255:red, 0; green, 0; blue, 0 }  ,fill opacity=1 ] (584.63,93) .. controls (584.63,91.99) and (585.45,91.17) .. (586.46,91.17) .. controls (587.47,91.17) and (588.29,91.99) .. (588.29,93) .. controls (588.29,94.01) and (587.47,94.83) .. (586.46,94.83) .. controls (585.45,94.83) and (584.63,94.01) .. (584.63,93) -- cycle ;
\draw  [fill={rgb, 255:red, 64; green, 54; blue, 255 }  ,fill opacity=1 ] (621.1,64.27) .. controls (621.1,63.25) and (621.92,62.43) .. (622.93,62.43) .. controls (623.95,62.43) and (624.77,63.25) .. (624.77,64.27) .. controls (624.77,65.28) and (623.95,66.1) .. (622.93,66.1) .. controls (621.92,66.1) and (621.1,65.28) .. (621.1,64.27) -- cycle ;
\draw  [dash pattern={on 4.5pt off 4.5pt}]  (622.93,64.1) -- (652.33,37.67) ;
\draw  [dash pattern={on 4.5pt off 4.5pt}] (342.96,144.17) .. controls (342.96,134.36) and (366.31,126.42) .. (395.13,126.42) .. controls (423.94,126.42) and (447.29,134.36) .. (447.29,144.17) .. controls (447.29,153.97) and (423.94,161.92) .. (395.13,161.92) .. controls (366.31,161.92) and (342.96,153.97) .. (342.96,144.17) -- cycle ;
\draw    (411.67,138.33) -- (461.67,122.33) ;
\draw  [fill={rgb, 255:red, 255; green, 0; blue, 0 }  ,fill opacity=1 ] (460.77,122.6) .. controls (460.77,121.59) and (461.59,120.77) .. (462.6,120.77) .. controls (463.61,120.77) and (464.43,121.59) .. (464.43,122.6) .. controls (464.43,123.61) and (463.61,124.43) .. (462.6,124.43) .. controls (461.59,124.43) and (460.77,123.61) .. (460.77,122.6) -- cycle ;
\draw  [fill={rgb, 255:red, 64; green, 54; blue, 255 }  ,fill opacity=1 ] (410.1,138.27) .. controls (410.1,137.25) and (410.92,136.43) .. (411.93,136.43) .. controls (412.95,136.43) and (413.77,137.25) .. (413.77,138.27) .. controls (413.77,139.28) and (412.95,140.1) .. (411.93,140.1) .. controls (410.92,140.1) and (410.1,139.28) .. (410.1,138.27) -- cycle ;

\draw (188,199.67) node [anchor=north west][inner sep=0.75pt]   [align=left] {$\displaystyle \mathbb{R}^{3}$};
\draw (375.33,199) node [anchor=north west][inner sep=0.75pt]   [align=left] {$\displaystyle \mathbb{R}^{2}$};
\draw (570.67,198.67) node [anchor=north west][inner sep=0.75pt]   [align=left] {$\displaystyle \mathbb{M}^{3}$};

\end{tikzpicture}
\captionof{figure}{The maps that induce the metrics whose lifted geodesic flows are the regularized Kepler flows on the respective energy levels. The sphere corresponds to negative energy levels and has radius $\frac{1}{\sqrt{2|c|}}$. The plane corresponds to the null energy level and the map is given by $x\mapsto \frac{2x}{|x|^2}$. The hyperboloid in Minkowski space $\mathbb{M}^3$ with signature $++-$ corresponds to positive energy levels and is defined by the equation $x^2 + y^2 - z^2 = -2c$.}
\end{center}\label{fig:metricmaps}

The energy levels $K=\frac{1}{2}$ are more precisely the \textit{spherical cotangent bundles} of these surfaces. We recall the spherical cotangent bundle of a Riemannian manifold $(M,g)$ is defined as
\begin{equation*}
 S^*M = \{\alpha \in T^*M \;|\; |\alpha|_{g^*} = 1\} \subseteq T^*M.
 \end{equation*}
 Putting these last remarks together, we reach a formulation of the well-known \emph{Moser-Osipov-Belbruno regularization}. It was Moser who first gave a regularization of this sort for negative energy levels in \cite{M70}. Osipov and Belbruno gave the corresponding positive energy regularization in \cite{O77} and \cite{B77} respectively. The regularizations can be condensed into one theorem covering all energy levels.

\begin{theorem}[Moser-Osipov-Belbruno]\label{thm:Moser}
The dynamics of the Kepler problem on the energy level $H=c$ are equivalent to the lifted geodesic flow on
\begin{itemize}
\item $S^*\mathbb{S}^2_c$ for $c < 0$,
\item $S^*\mathbb{R}^2$ for $c = 0$ and
\item $S^*\mathbb{H}^2_c$ for $c > 0 $
\end{itemize}
after applying the stereographic maps described in Figure \ref{fig:metricmaps}. 
\end{theorem}
For an elementary geometric proof of this result, we recommend Geiges' exposition \cite{G16}.

\subsection*{The contact form and adapted metric}
We have already seen in Lemma \ref{lemma:mechanicalcontact} that the contact form for which the regularized Kepler flow is Reeb-like on an energy level is the Liouville form $\alpha = ydx$ restricted to the energy level.  If we want the contact form on the original phase space, we simply pull back the Liouville form by the maps that induce the metrics and then by the symplectic switch. However, we continue the exposition in the context of constant curvature surfaces, as this results in more elegant and concise formulations of the statements.

The following proposition is well known and gives an equivalent characterization of adapted metrics to Reeb vector fields. We will use it to obtain the adapted metric to the Kepler Reeb vector field explicitly.

\begin{proposition}\label{prop:adaptedmetric}
Let $(\alpha,R)$ be a Reeb field. A Riemannian metric $g$ such that $R$ is $g$-orthogonal to $\ker\alpha$ and $|R|^2_g$ is constant is an adapted metric to $(\alpha,R)$.
\end{proposition}

\begin{proof}
Since $\ker \iota_R g = \ker \alpha$, necessarily $\iota_R g = h\alpha$ for some $h\neq 0$. Furthermore, $|R|^2_g = h\alpha(R) = h$ is constant by hypothesis, say $h=1$. Taking $\mu = \alpha \wedge (d\alpha)^{n}$, which is a volume form preserved by $R$, we obtain
\begin{equation*}
\iota_{R}\mu = (d\alpha)^{n} = (d\iota_R g)^{n},
\end{equation*}
so that $\curl R = R$, as we wanted to prove. 
\end{proof}

\begin{theorem}\label{cotangent lift}
Let $(M,g)$ be a Riemannian manifold. The Reeb field on $S^*M$ with respect to the canonical Liouville form $\alpha$ is orthogonal to $\ker\alpha$ and of constant magnitude with respect to the canonical lift $S^*g$ of $g$ to the spherical cotangent bundle.
\end{theorem}

Before the proof of this theorem, we recall how to construct $S^*g$. We begin by emphasizing that $g$ is a tensor field that takes vectors of $TM$, while $S^*g$ is a tensor field that takes vectors of $T(S^*M)$. We construct $S^*g$ by first constructing the cotangent lift $T^*g$ on $T^*M$ and then restricting it to $S^*M$. The cotangent lift was introduced by Mok in \cite{M77} following Sasaki's construction of the tangent lift metric in \cite{S58}. 

The tangent bundle of the cotangent bundle $T^*M \stackrel{\pi}{\rightarrow}M$ splits on each $\xi\in T^*M$ into the \emph{vertical tangent space} $V_\xi = \ker T_{\xi}\pi$ and a \emph{horizontal tangent space} $H_\xi$, which without additional structure cannot be chosen canonically. It is only required that $T_\xi(T^*M) = V_\xi \oplus H_\xi$. In fact, a smooth choice of horizontal tangent spaces is equivalent to a choice of a connection on $M$. Precisely for this reason, the metric $g$ provides a canonical choice of $H$ through the Levi-Civita connection,
\begin{equation*}
H_\xi = \{(\dot{\gamma}, \dot{\theta}) \in T_\xi(T^*M) \; | \; (\gamma(t),\theta(t)) \subseteq T^*M \; \text{and } \nabla_{\dot{\gamma}}\theta = 0\}.
\end{equation*}
It is easy to check that indeed $T_\xi(T^*M) = V_\xi \oplus H_\xi$ and that $H_\xi$ is naturally isomorphic to $T_{\pi(\xi)}M$, since $H_\xi$ is identified with the space of geodesics on $M$ passing through $\pi(\xi)$. On the other hand, $V_\xi \cong T^*_{\pi(\xi)}M$, because if we take paths of the form $(\pi(\xi),t\theta)$, their tangent fields are in $V_\xi$ and they are identified with $\theta\in T^*_{\pi(\xi)}M$. The following figure illustrates this splitting.

\begin{center}

\tikzset{every picture/.style={line width=0.75pt}} 

\begin{tikzpicture}[x=1pt,y=1pt,yscale=-1,xscale=1]

\draw    (134,94) .. controls (202,107) and (238,172) .. (230.5,197) ;
\draw    (331,57) .. controls (374,76) and (401,162.5) .. (401,192.5) ;
\draw    (134,94) .. controls (229,99) and (318,90) .. (331,57) ;
\draw    (230.5,197) .. controls (275.5,189) and (357,188.5) .. (401,192.5) ;
\draw   (227.87,108.84) -- (318.67,94.23) -- (362.47,155.39) -- (271.67,170) -- cycle ;
\draw    (279.67,132.87) -- (305.09,113.69) ;
\draw [shift={(307.48,111.89)}, rotate = 142.97] [fill={rgb, 255:red, 0; green, 0; blue, 0 }  ][line width=0.08]  [draw opacity=0] (8.04,-3.86) -- (0,0) -- (8.04,3.86) -- (5.34,0) -- cycle    ;
\draw [color={rgb, 255:red, 255; green, 0; blue, 0 }  ,draw opacity=1 ]   (279.67,132.87) -- (303.06,132.07) ;
\draw [shift={(306.06,131.97)}, rotate = 178.05] [fill={rgb, 255:red, 255; green, 0; blue, 0 }  ,fill opacity=1 ][line width=0.08]  [draw opacity=0] (5.36,-2.57) -- (0,0) -- (5.36,2.57) -- (3.56,0) -- cycle    ;
\draw [color={rgb, 255:red, 37; green, 0; blue, 250 }  ,draw opacity=1 ]   (307.48,111.89) -- (287.68,107.19) ;
\draw [shift={(284.76,106.5)}, rotate = 13.34] [fill={rgb, 255:red, 37; green, 0; blue, 250 }  ,fill opacity=1 ][line width=0.08]  [draw opacity=0] (5.36,-2.57) -- (0,0) -- (5.36,2.57) -- (3.56,0) -- cycle    ;
\draw [color={rgb, 255:red, 37; green, 0; blue, 250 }  ,draw opacity=1 ]   (322.5,178) -- (316.15,176.28) ;
\draw [shift={(313.26,175.5)}, rotate = 15.13] [fill={rgb, 255:red, 37; green, 0; blue, 250 }  ,fill opacity=1 ][line width=0.08]  [draw opacity=0] (5.36,-2.57) -- (0,0) -- (5.36,2.57) -- (3.56,0) -- cycle    ;
\draw [color={rgb, 255:red, 255; green, 0; blue, 0 }  ,draw opacity=1 ]   (331.67,176.87) -- (338.53,175.91) ;
\draw [shift={(341.5,175.5)}, rotate = 172.07] [fill={rgb, 255:red, 255; green, 0; blue, 0 }  ,fill opacity=1 ][line width=0.08]  [draw opacity=0] (5.36,-2.57) -- (0,0) -- (5.36,2.57) -- (3.56,0) -- cycle    ;

\draw (207.75,99) node [anchor=north west][inner sep=0.75pt]  [font=\tiny] [align=left] {$\displaystyle T_{\pi ( \xi )} M$};
\draw (284.15,115.46) node [anchor=north west][inner sep=0.75pt]  [font=\tiny] [align=left] {$\displaystyle \xi $};
\draw (284.88,137.24) node [anchor=north west][inner sep=0.75pt]  [font=\tiny] [align=left] {$\displaystyle v \in H_{\xi }$};
\draw (292.38,99.24) node [anchor=north west][inner sep=0.75pt]  [font=\tiny] [align=left] {$\displaystyle w \in V_{\xi }$};
\draw (293.88,171.24) node [anchor=north west][inner sep=0.75pt]  [font=\tiny] [align=left] {$\displaystyle W =\ \ \ +\ \ \ \ \in T_{\xi }\left( T^{*} M\right) \ $};

\end{tikzpicture}\label{HorizontalAndVertical}
\captionof{figure}{Splitting of $W = (v,w) \in T_\xi(T^*M)$ in components $v$ of a horizontal tangent space and $w$ of the vertical tangent space. The horizontal component $v$ can be thought of as moving the base point $\pi(\xi)$ while the vertical component $w$ moves the covector $\xi$ within the cotangent space at $\pi(\xi)$.}
\end{center}
With these natural identifications we finally obtain
\begin{equation*}
 T_\xi(T^*M) \cong T_{\pi(\xi)}^*M \oplus T_{\pi(\xi)}M,
\end{equation*}
and since there are natural inner products on each component, given by $g^*$ and $g$ respectively, we can define an inner product on $T_\xi(T^*M)$ by imposing that the components are orthogonal to each other.

A coordinate expression for the resulting metric is the following. If $V = (\dot{\alpha}, \dot{\theta})$ and $W = (\dot{\beta}, \dot{\omega})$,
\begin{equation*}
T^*g(V,W) = g(T\pi(V),T\pi(W)) + g^*(\nabla_{\dot{\alpha}} \theta, \nabla_{\dot{\beta}} \omega).
\end{equation*}
The first component projects onto $H_\xi \cong T_{\pi(\xi)}M$ and the second onto $V_\xi \cong T^*_{\pi(\xi)}M$.

\begin{proof}[Proof of Theorem \ref{cotangent lift}]
We first check that $R$ is orthogonal to $\ker\alpha$ with respect to $S^*g$, which, again, is $T^*g$ restricted to $T(S^*M)$. Taking natural coordinates $(x,y)$, we have
\begin{equation*}
K(x,y) = \frac{1}{2}g^{ij}(x)y_iy_j \Rightarrow dK(x,y) = \frac{1}{2}g^{ij}_{,k}y_iy_jdx^k + g^{ij}y_idy_j,
\end{equation*}
where we are using Einstein notation to avoid notational clutter. In this case, the Reeb vector field is the Hamiltonian vector field, so it has the following expression,
\begin{equation*}
R = X_K = \frac{1}{2}g^{ij}(x)_{,k}y_iy_j\frac{\partial}{\partial y_k} - g^{ij}y_i\frac{\partial}{\partial x^j}.
\end{equation*}
On the other hand, a vector field $Y\in \ker \alpha$ is necessarily of the form $Y = \alpha^{i}\frac{\partial}{\partial x^{i}} + \beta_j\frac{\partial}{\partial y_j}$ with
\begin{equation*}
\alpha(Y) = y_kdx^k(\alpha^{i}\frac{\partial}{\partial x^{i}} + \beta_j\frac{\partial}{\partial y_j}) = y_i\alpha^{i} = 0.
\end{equation*}
Thus, since the projection of $R$ onto the horizontal tangent spaces vanishes for being a geodesic vector field, the inner product of $R$ and $Y$ by $S^*g$ is
\begin{equation*}
S^*g(R,Y) = g(T\pi(R),T\pi(Y)) = g(- g^{ij}y_i\frac{\partial}{\partial x^j}, \alpha^{i}\frac{\partial}{\partial x^{i}}) = -g_{kl}g^{ik}y_i\alpha^l = -y_i\alpha^{i} = 0.
\end{equation*}
Therefore $R$ is orthogonal to $\ker\alpha$. Lastly, for the constant magnitude condition we see that
\begin{equation*}
S^*g(R,R) = g(- g^{ij}y_i\frac{\partial}{\partial x^j}, - g^{ij}y_i\frac{\partial}{\partial x^j}) = g_{kl}g^{ik}y_ig^{jl}y_j = g^{ij}y_iy_j = 1
\end{equation*}
for being always on $S^*M$.
\end{proof}

Combining Proposition \ref{prop:adaptedmetric} and Theorem \ref{cotangent lift}, we finally get an interpretation of the Kepler problem as a stationary Beltrami solution to the Euler equations.
\begin{corollary}[The Kepler-Euler flow]\label{cor:keplereuler}
The regularized Kepler flow on the $c$-energy level is a stationary Beltrami solution to the Euler equations on
\begin{itemize}
\item $S^*\mathbb{S}^2_c$ if $c < 0$,
\item $S^*\mathbb{R}^2$ if $c = 0$ and
\item $S^*\mathbb{H}^2_c$ if $c > 0$
\end{itemize}
with the liftings to the spherical cotangent bundles of the respective constant $(-2c)$-curvature metrics. The flow lines are lifted geodesics. The Kepler flow on the plane is recovered from the natural stereographic projections of the respective surfaces, or from the involution $x\mapsto\frac{2x}{|x|^2}$ when $c=0$.
\end{corollary}

\subsection*{Further remarks on the Kepler-Euler flow}

We conclude with a more detailed description of the lifted metrics and dynamics of the Kepler-Euler flow, relating them to other known Beltrami fields when possible. 

\textbf{The metrics.} The coordinate expressions for the spherical cotangent lift metrics are rather long and messy, so we do not give them here. However, the spherical tangent lift of the round sphere metric was studied by Klingenberg and Sasaki \cite{KS75}, and the lifts of general metrics on surfaces were studied by Nagy \cite{N77}. Since the tangent and cotangent bundles (with the lifted metrics) are isometric~\cite{M77}, the results therein also apply to our case. In the latter article, Nagy showed that the spherical tangent bundle of a surface is of constant curvature only when the base surface is of constant curvature 0 or 1. In these cases, the curvature of the spherical tangent bundles are 0 and $\frac{1}{4}$ respectively. In particular, the Kepler-Euler flow on energy level $c$ takes place on a constant curvature 3-manifold exactly when $c=0$ or $c=-\frac{1}{2}$.

Sasaki further classified geodesics on the spherical tangent bundles of space forms into \textit{horizontal}, \textit{vertical}, and \textit{oblique} geodesics \cite{S76}. Horizontal geodesics are geodesics in which the unit tangent vector is parallel transported along the projection of the geodesic to the base manifold; vertical geodesics are rotations along the fibers of the sphere bundle, the base point remaining stationary; oblique geodesics have some combination of motion through the base and fibers. Since the Kepler-Euler flow is the lift to the spherical cotangent bundle of base geodesics, all of the trajectories are horizontal geodesics according to Sasaki's classification.

\textbf{The Beltrami vector fields.} First, we give the Hamiltonian vector field for the regularized Kepler Hamiltonian given in Equation (\ref{eq:regularizedKepler}). In those coordinates, which are the stereographic coordinates $(x_1,x_2,\alpha)$ on $S^*_c\mathbb{R}^{2}$, where the spherical cotangent bundle is taken with respect to the metric corresponding to the $c$-energy level, the vector field is expressed as
\begin{equation*}
    X_c = \frac{|x|^2 - 2c}{2}( \cos\alpha \partial_{x_1}+\sin\alpha\partial_{x_2} ) + (x_1\sin\alpha - x_2\sin\alpha)\partial_\alpha.
\end{equation*}
A straightforward computation shows that for all $c$, these Beltrami fields have eigenvalue 1, that is, $X_c = \curl_c X_c$, where $\curl_c$ is the curl with respect to the lift of the corresponding constant curvature metric. We now give a few more details on the dynamics of these Beltrami fields according to the sign of the energy level.

When $c<0$, in spherical coordinates $(\phi,\theta,\alpha)$ for $S^*\mathbb{S}^{2}_c\cong\mathbb{R}\mathbb{P}^3$, the vector field is expressed as
\begin{equation*}
    X_c = \rho\sqrt{-2c}\left(\frac{\cos\alpha}{\sin\theta}\partial_\phi + \sin\theta\sin\alpha\partial_\theta + \frac{\cos\theta}{\sin^2\theta}\cos^3\alpha\partial_\alpha\right),
\end{equation*}
where $\rho = \left(\cos^2\alpha + \sin^2\theta\sin^2\alpha\right)^{-\frac{1}{2}}$. While this expression may be somewhat messy, the dynamics are clear: For all $c<0$, the flow is simply a constant rescaling of the flow on $c = -\frac{1}{2}$, and it is well known that the lift of the geodesic flow to $S^*\mathbb{S}^{2}\cong \mathbb{R}\mathbb{P}^3$ is the quotient of the Hopf flow on $\mathbb{S}^{3}$ by the antipodal action (see, for example, Chapter 4 in \cite{FK18}). Thus, the orbits are all the Hopf orbits under the antipodal map quotient.

When $c = 0$, applying the involution $x\mapsto \frac{2x}{|x|^2}$, we can reexpress $X_0$ as
\begin{equation*}
    X_0 = \cos\alpha\partial_{x_1} + \sin\alpha\partial_{x_2}.
\end{equation*}
This flow on $S^*\mathbb{R}^{2}\cong \mathbb{R}^2\times\mathbb{S}^{1}$ is a \textit{shear} type Beltrami flow (meaning that the flow consists of parallel stream lines that twist as the angle $\alpha$ changes).

Finally, when $c>0$, in half-plane coordinates $(x,y,\alpha)$ for $S^*\mathbb{H}^2_c\cong\mathbb{R}\times\mathbb{R}^+\times\mathbb{S}^{1}$ the vector field is expressed as
\begin{equation*}
    X_c = y\sqrt{2c}\left(\cos\alpha\partial_x + \sin\alpha\partial_y\right) + \sqrt{2c}\cos\alpha\partial_\alpha.
\end{equation*}
As in the case $c<0$, this Beltrami field is the same for all values of $c$ up to a constant rescaling, so the dynamics are the same. Naturally, all streamlines are open and tend to $y=0$ or $y = +\infty$. 


\end{document}